\documentclass[12pt]{amsproc}%
\usepackage{amsfonts}
\usepackage{amsmath}
\usepackage{amssymb}
\usepackage{graphicx}%
\renewcommand {\theequation} {\@arabic\c@equation}
 \theoremstyle{plain}

\newtheorem{corollary}{Corollary}

\newtheorem{definition}{Definition}

\newtheorem{lemma}{Lemma}

\newtheorem{proposition}{Proposition}
\newtheorem{remark}{Remark}

\newtheorem{theorem}{Theorem}
\numberwithin{equation}{section}
\begin{document}
\centerline{\Large{\bf An analogue of Szeg$\ddot{o}$'s limit
theorem}}

\vspace{0.3cm}

 \centerline{\Large{\bf in free probability theory }}

\vspace{1cm}

\centerline{Junhao Shen\footnote{The second author is supported   by
an NSF grant.}}
\bigskip

\centerline{\small {Department of Mathematics  and Statistics,
University of New Hampshire, Durham, NH, 03824}}

\vspace{0.2cm}

\centerline{
Email: \qquad jog2@cisunix.unh.edu \qquad }

\bigskip

\noindent\textbf{Abstract: } In the paper, we discuss   orthogonal
polynomials in free probability theory. Especially, we prove an
analogue of of Szeg$\ddot{o}$'s limit theorem in free probability
theory.

\vspace{0.2cm} \noindent{\bf Keywords:} orthogonal polynomial,
Szeg$\ddot{o}$'s limit theorem, free probability

\vspace{0.2cm} \noindent{\bf 2000 Mathematics Subject
Classification:} Primary  42C05, Secondary  46L10

\section{Introduction}

Szeg$\ddot{o}$'s limit theorem plays an important role in the theory
of orthogonal polynomials in one variable  (see
\cite{Deift},\cite{Simon}). Given a real random variable  $x$   with
a compact support in a probability space, then  Szeg$\ddot{o}$'s
limit theorem (see for example \cite{Simon}) provides us the
information of asymptotic behavior of  determinants of Toeplitz (or
Hankel) matrices associated with $x$ (equivalently the asymptotic
behavior of   volumes of the parallelograms spanned by $1, x,\ldots,
x^q$).

 The theory
of free probability was developed by Voiculescu from 1980s  (see
\cite {Voi}). One basic concept in free probability theory is
``freeness", which is the analogue of ``independence" in probability
theory. The purpose of this paper is to study  Szeg$\ddot{o}$'s
limit theorem in the context of  free probability theory.

Suppose $\langle \mathcal M, \tau\rangle$ is a free probability
space and $x_1,\ldots, x_n$ are   random variables in $\mathcal M$
such that $x_1,\ldots,x_n$ are free with respect to $\tau$. Our
result (Theorem 1) in the paper, as an analogue of Szeg$\ddot{o}$'s
limit theorem,  describes the asymptotic behavior of   determinants
of the Hankel matrices associated with $x_1,\ldots, x_n$. More
specifically, we proved that the following equation:
$$
\lim_{q\rightarrow\infty} \frac {\ln D_{q+1}(x_1,\ldots,x_n)}{q\cdot
n^{q}} = \frac {(n-1)}{n } \cdot \sum_{k=1}^n\mathcal E_n(x_k),
$$
where $D_{q+1}(x_1,\ldots,x_n) $ is the Hankel determinant
associated with $x_1,\ldots, x_n$ (see Definition 3); and $\mathcal
E_n(x_k)$ is $n-$th entropy number of $x_k$ (see Definition 5).

 The
organization of the paper is as follows. We review the process of
Gram-Schmidt orthogonalization in section 2. Generally, orthogonal
polynomials can   be computed by Gram-Schmidt orthogonalization. In
section 3, we introduce   families of orthogonal polynomials in
several noncommutative variables and the concept of Hankel
determinant. The relationship between Hankel determinant and
volume of the parallelogram spanned a family of vectors is also
mentioned in this section. The Szeg$\ddot {o}$'s limit theorem in
one variable is recalled in section 4. We state and prove the main
Theorem, as an analogue of Szeg$\ddot {o}$'s limit theorem in free
probability theory, in section 5.

\section{Gram-Schmidt Orthogonalization}
In this section, we will review the process of Gram-Schmidt
orthogonalization. Suppose $H$ is a complex Hilbert space. Let
$\{y_q\}_{q=1}^N$ be a family of linearly independent vectors in
$H$, where $N$ is a positive integer or infinity. Let, for each
$2\le q\le N$, $H_q$ be the closed subspace linearly spanned by
$\{y_1, \ldots, y_{q-1}\}$ in $H$. Let  $E_1=0$ and $E_q$ be the
projection from $H$ onto $H_{q}$ for $2\le q\le N$. Then, for each
$1\le q\le N$, we have
$$
y_q-E_q(y_q) = \frac 1{D_{q}} \left |
  \begin{aligned}
  &\langle y_1, y_1\rangle  & \quad \langle y_2, y_1\rangle  & \quad \cdots
  &  \quad \langle y_q, y_1\rangle \\
 &\langle y_1, y_2\rangle  & \quad \langle y_2, y_2\rangle  & \quad \cdots
  &  \quad \langle y_q, y_2 \rangle \\
   &  & & \quad  \cdots   &\\
    &\langle   y_1, y_{q-1}\rangle  & \quad \langle  y_2, y_{q-1}  \rangle  & \quad \cdots
  &  \quad \langle  y_q, y_{q-1}\rangle \\
   & \ \ \ y_1  & \quad  y_2   & \quad \cdots
  &  \quad y_q \\
  \end{aligned}
  \right |,
$$
where, for $q\ge 1$
$$
D_{q+1}=|(\langle y_j, y_i\rangle)_{1 \le i,j\le q}|=\left |
  \begin{aligned}
  &\langle y_1, y_1\rangle  & \quad \langle y_2, y_1\rangle & \quad \cdots
  &  \quad \langle y_q, y_1\rangle \\
 &\langle y_1, y_2\rangle  & \quad \langle y_2, y_2\rangle & \quad \cdots
  &  \quad \langle y_q, y_2\rangle \\
   &  &  & \quad  \cdots   &\\
    &\langle y_1, y_{q-1} \rangle  & \quad \langle y_2, y_{q-1}\rangle  & \quad \cdots
  &  \quad \langle y_q, y_{q-1} \rangle \\
  &\langle y_1, y_{q}\rangle  & \quad \langle y_2, y_{q}\rangle  & \quad \cdots
  &  \quad \langle y_{q}, y_q\rangle
  \end{aligned}
  \right |
$$ and $D_1=1$.

The following proposition follows easily from the process of
Gram-Schmidt orthogonalization.
\begin{proposition}
For each $1\le q\le N$, we have
$$\begin{aligned}
 D_{q+1}&=\prod_{i=1}^q \|y_i-E_i(y_i)\|_2^2= \prod_{i=1}^q \langle
y_i-E_i(y_i), y_i-E_i(y_i) \rangle \\
& = (\text{volume of the  parallelogram linearly spanned by
$y_1,\ldots, y_q$ in $H$})^2.\end{aligned}
$$
\end{proposition}

\section{Definitions of Orthogonal Polynomials in Free Probability}

A pair of objects $\langle\mathcal M,\tau\rangle$ is called a free
probability space when $\mathcal M$ is a finite von Neumann algebra
and $\tau$ is a faithful normal tracial state on $\mathcal M$ (see
\cite{Voi}). Let $H$ be the complex Hilbert space $L^2(\mathcal
M,\tau)$. Let $x_1,\ldots, x_n$ be a family of
 random variables in $\mathcal M$. Let $\mathcal A(x_1,\ldots,x_n)$ be the
unital algebra consisting of non-commutative polynomials of $I, x_1,
\ldots,x_n$ with complex coefficients, where $I$ is the identity
element of $\mathcal M$.

\begin{definition}
Suppose $\Sigma$ is a  totally ordered index set.  Then
 $\{P_\alpha(x_1,\ldots,x_n)\}_{\alpha\in \Sigma}$  in $\mathcal A(x_1,\ldots,x_n)$
is called a family of  orthogonal polynomials of $x_1,\ldots, x_n$
in $\mathcal M$ if, for all $\alpha,\beta$ in $\Sigma$ with
$\alpha\ne \beta$,
  $\tau(P_\beta(x_1,\ldots,x_n)^* P_\alpha(x_1,\ldots,x_n))=0$.

\end{definition}



\subsection{Orthogonal polynomials in one variable}  Suppose $x
$ is an element in $\mathcal M$. Let $H_0=\Bbb CI$ and $H_q$ be the
linear subspace spanned by the elements $\{I, x,
x^2,\ldots,x^{q-1}\}$ in $H$ for each $q\ge 2$. Let $E_q$ be the
projection from $H$ onto $H_{q}$.

For each $q$ in $\Bbb N$, we  let $P_q(x)$ be $x^q-E_q(x^q)$,
obtained by the process of Gram-Schmidt orthogonalization as in
section 2. It is not hard to see that $\{P_q(x)\}_{q\in\Bbb N}$ is a
family of orthogonal polynomials of $x$ in $\mathcal M$.

\subsection{Recursive formula of orthogonal polynomials in one variable on the real line}
The following recursive formula is well-known. (see \cite{Deift} or
\cite {Simon})
\begin{lemma}
Suppose $x$ is a self-adjoint element in a free probability space $
\langle \mathcal M,\tau\rangle $ and $\{P_q(x)\}_{q=1}^\infty$ is
defined as in section 3.1. Then there are sequences of real numbers
$\{a_q\}_{q=1}^\infty$ with $a_q>0$ $(\forall q \ge 1)$ and
$\{b_q\}_{q=1}^\infty$ such that
$$
xP_q(x)= P_{q+1 }(x) + b_{q+1}P_q(x) + a_{q}^2P_{q-1}(x), \qquad
\text { for all  } q\ge 2 .
$$ These $a_{ 1}, a_{ 2},\ldots$ are called the coefficients of Jacobi
matrix associated with $x $. Moreover,
$$
\|P_q(x) \|_2=\tau(P_q(x)^*P_q(x))^{1/2}= a_1a_2\cdots a_q, \qquad
\text { for all $ q\ge 2$.}
$$
\end{lemma}

\subsection{Recursive formula of orthogonal polynomials in one variable on the unit circle}

Suppose $u$ is a unitary element in a free probability space $
\mathcal M $ and $\{P_q(u)\}_{q=1}^\infty$ is defined as in section
3.1. For each $q\ge 1$, let $K_q$ be the linear subspace spanned by
the set $\{u, u^2,\ldots, u^q\}$ in $H$. Let $F_q$ be the projection
from $H$ onto $K_q$. Define $Q_q(u)= I-F_q(I)$ for each $q\ge 1$
(see Lemma 1.5.1 in \cite{Simon}).

With the notations as above. We have the following result. (see
\cite{Deift} or \cite {Simon})
\begin{lemma}
There exists a sequence of complex number
$\{\alpha_q\}_{q=0}^\infty$ with $|\alpha_q|\le 1 $ ($\forall q\ge
1$) so that
$$
\begin{aligned}
 P_{q+1}(u)&=uP_q(u)-\overline{\alpha}_q Q_(u)\\
Q_{q+1}(u)&=Q_q(u)-\alpha_q u P_q(u).
\end{aligned}
$$ These $\alpha_1,\alpha_2,\ldots $ are called
Verblunsky coefficients.  Moreover,
$$
\|P_q(u)\|^2= \prod_{j=0}^q(1-|\alpha_q|^2).
$$
\end{lemma}
\subsection{Orthogonal polynomials in several variables}  Suppose $x_1,\ldots,x_n $ is a family of
  elements in $\mathcal M$.

Let $\Sigma_n=\Bbb F_n^+ $ be the unital free semigroup generated by
$n$ generators $X_1,\ldots, X_n$ with lexicographic order $\prec$,
i.e.
$$
e\prec X_1\prec X_2\prec \ldots \prec X_n\prec X_1^2\prec
X_1X_2\prec X_1X_3\prec \ldots\prec x_n^2\prec X_1^3\prec \cdots
$$

For each $\alpha=X_{i_1}X_{i_2}\cdots X_{i_q}$ in $\Sigma_n$ with
$q\ge 1$, $1\le i_1,\ldots, i_q\le n$, we define
$$x_\alpha=x_{i_1}x_{i_2}\cdots x_{i_q}\qquad \text { in } \mathcal M.$$
We also define $x_e=I$.

Let $H_e=\Bbb CI$. Note that each element in $\mathcal M$ can  be
canonically identified  as a vector in $H$. For each $\alpha\in
\Sigma_n$, let $H_\alpha$ be the linear subspace spanned  by the set
$\{x_\beta\}_{\beta\prec\alpha}$ in $H$. Let $E_\alpha$ be the
projection from $H$ onto the closure of the subspace
$\cup_{\beta\prec \alpha}H_{\beta}$.

For each $\alpha$ in $\Sigma_n$, we  let $P_\alpha(x_1,\ldots,x_n)$
be $x_\alpha-E_\alpha(x_\alpha)$, obtained by the process of
Gram-Schmidt orthogonalization on the family of vectors
$\{x_\beta\}_{\beta\prec \alpha}$ in $H$. It is not hard to see that
$$\{P_\alpha(x_1,\ldots,x_n)\}_{\alpha\in\Sigma}$$ is a   family of
orthogonal polynomials of $x_1,\ldots,x_n$ in $\mathcal M$.

\begin{remark}
If $x_1,\ldots, x_n$ are algebraically free (i.e. satisfy no
algebraic relation), then $P_\alpha(x_1,\ldots,x_n)\ne
P_\beta(x_1,\ldots,x_n)$ for all $\alpha\ne \beta$ in $\Sigma_n$ and
$P_\gamma\ne 0$ for all $\gamma$ in $\Sigma$.
\end{remark}

\begin{definition}
Let $\Sigma_n=\Bbb F_n^+ $ be the unital free semigroup generated by
$n$ generators $X_1,\ldots, X_n$ with lexicographic order $\prec$.
For every $\alpha $ in $\Sigma_n$, we define the length of $\alpha$
as
$$
|\alpha|=\left \{ \begin{aligned}
& 0, \qquad \text { if } \alpha =e\\
& q, \qquad \text { if } \alpha =X_{i_1}X_{i_2}\cdots X_{i_q} \text
{  for some $q\ge 1$, $1\le i_1,\ldots, i_q\le n$}
\end{aligned}  \right.
$$
\end{definition}

\subsection{Hankel Determinant}
\begin{definition}
For each $\gamma\in \Sigma_n$, let $m$ be the cardinality of the set
$\{x_\alpha\}_{\alpha\prec \gamma}$ and $A_\gamma$   be an $m\times
m$ complex matrix  such that $(\alpha, \beta)-$th entry of
$A_\gamma$ is equal to $\tau(x_\alpha^*x_\beta )$ for each
$\alpha,\beta\prec \gamma$. Define the Hankel determinant  $\bar
D_\gamma(x_1,\ldots,x_n)$ to be   the determinant of $A_\gamma$,
i.e.
$$
\bar D_\gamma(x_1,\ldots,x_n) = |A_\gamma|=
|(\tau(x_\alpha^*x_\beta))_{\alpha,\beta\prec \gamma} |.
$$

For each $q\ge 1$, let $k$ be the cardinality of the set
$\{x_\alpha\}_{|\alpha|< q}$ $($where $|\alpha|$ is as defined in
Definition 3$)$ and $A_q$   be a $k\times k$ complex matrix  such
that $(\alpha, \beta)-$th entry of $A_q$ is equal to
$\tau(x_\alpha^*x_\beta)$ for each $\alpha,\beta$ with
$|\alpha|,|\beta|< q$.  Define the {Hankel Determinant} $
D_q(x_1,\ldots,x_n)$ to be the determinant of $A_q$, i.e.
$$
  D_q(x_1,\ldots,x_n) = |A_q|=
|(\tau(x_\alpha^*x_\beta))_{|\alpha|,|\beta|< q} |.
$$
\end{definition}

The following proposition follows directly from Proposition 1 and
Definition 4.
\begin{proposition}
$$
\begin{aligned}
  \bar D_{\gamma}(x_1,\ldots,x_n) &= \prod_{\alpha \prec \gamma} \|P_\alpha(x_1,\ldots,x_n)\|_2^2,
  \quad \text { for all } \gamma \in \Sigma_n\\
   & =(\text{volume of the  parallelogram linearly spanned by
$\{x_\alpha\}_{\alpha\prec \gamma}$ in $H$})^2\\
   D_{q}(x_1,\ldots,x_n) &= \prod_{|\alpha|<q} \|P_\alpha(x_1,\ldots,x_n)\|_2^2,
  \quad \text { for all } q \in \Bbb N.\\
  & =(\text{volume of the  parallelogram linearly spanned by
$\{x_\alpha\}_{|\alpha|<q}$ in $H$})^2
\end{aligned}
$$\end{proposition}

\section{Szeg$\ddot{o}$'s Limit Theorem in One Variable}

In this section, we will recall Szeg$\ddot{o}$'s Limit Theorem.
\subsection{Szeg$\ddot{o}$'s functions of class $G$} Let $G$ denote
the class of functions $w(t)\ge 0$, defined and measurable in
$[-1,1]$, for which the integrals
$$
\int_{-\pi}^\pi w(\cos \theta)|\sin\theta| d\theta, \qquad  \qquad
\int_{-\pi}^\pi |\log(w(\cos \theta)|\sin\theta|)| d\theta
$$
exist with the first integral supposed positive.
\subsection{Szeg$\ddot{o}$'s Limit Theorem} We state Szeg$\ddot{o}$'s Limit
Theorem as follows. (see \cite{Sze})
\begin{lemma}
Suppose $\mathcal M$ is a free probability space with a  tracial
state $\tau$. Let $x$ be a self-adjoint random variable in $\mathcal
M$ with density function $w(t)$ defined on $[-1,1]$, i.e.
$$
\tau(x^q)= \int_{-1}^1 t^q w(t) dt, \qquad \text { for all } q\ge 1.
$$
Suppose $P_q(x)$ for $q=1, 2, 3,\cdots$ are the orthogonal
polynomials as defined in subsection 3.1. If the function $w(t)$
belongs to class $G$, then as $q\rightarrow\infty$,
$$
 \|P_q(x)\|_2 \backsimeq \pi^{-1/2}2^q e^{^{-\frac 1 {2\pi} \int_{-1}^1
\log w(t) \frac {dt}{\sqrt{1-t^2}}}}.
$$
\end{lemma}

\begin{remark}
Combining Lemma 3 and Proposition 1, we obtain the information of
asymptotic behavior of  $D_q$, which is the determinant  of a
Toeplitz matrix   when $x$ is a unitary element, or the determinant
of a Hankel matrix when $x$ is a self-adjoint element.

\end{remark}

\section{An Anolague of Szeg$\ddot{o}$ Limit Theorem in Free Probability Theory}

In this section, we will follow the previous notations. Let
$\langle\mathcal M,\tau\rangle$ be a free probability space. Let
$x_1,\ldots, x_n$ be a family of random variables in $\mathcal M$.

For each $k, 1\le k\le n$, and integer $q\ge 1$, we let
$P_{k,q}(x_k)$ be $x_k^q-E_{k,q}(x_k^q)$ where $E_{k,q}$ is the
projection from the Hilbert space $L^2(\mathcal M,\tau)$ onto the
linear subspace spanned by $\{I, x_k, \ldots, x_k^{q-1}\}$ in
$L^2(\mathcal M,\tau)$. By section 3.1,
$\{P_{k,q}(x_k)\}_{q=1}^{\infty}$ is the family of orthogonal
polynomials associated with $x_k$ in $\mathcal M$.

 \begin{definition}{ (See \cite {Voi})
      The von Neumann subalgebras $\mathcal M_i,
    i \in \mathcal I$ of $\mathcal M$ are {\em free} with respect to
    the trace $\tau$ if $\tau (y_1\ldots y_n)=0$ whenever $y_j \in
    \mathcal M_{i_j}, i_1\ne \ldots \ne i_n$ and $\tau (y_j)=0$ for
    $1 \le j \le n$ and every $n$ in $\Bbb N$.
     (Note that $i_1$ and $i_3$, for example, may be equal:
     ``adjacent" $A_i$s are not in the same $\mathcal M_i$). A
     family of self-adjoint elements $\{x_1, \ldots, x_n\}$ is
     free with respect to the trace $\tau$ if the von Neumann
     subalgebras $\mathcal M_i$ generated by the $x_i$ are
     free with respect to the trace $\tau$.

    }
    \end{definition}

\subsection{A few lemmas}Let $\Sigma_n=\Bbb F_n^+ $ be the unital free semigroup
generated by $n$ generators $X_1,\ldots, X_n$ with lexicographic
order $\prec$. For each $\alpha=X_{i_1}X_{i_2}\cdots X_{i_q}$ in
$\Sigma_n$ with $q\ge 1$, $1\le i_1,\ldots, i_q\le n$, we define
$$x_\alpha=x_{i_1}x_{i_2}\cdots x_{i_q}\qquad \text { in } \mathcal M.$$ We let
$P_\alpha(x_1,\ldots, x_n)$ be $x_\alpha-E_\alpha(x_\alpha)$ where
$E_\alpha$ the projection from $L^2(\mathcal M,\tau)$ onto the
linear subspace spanned by $\{x_{\beta}\}_{\beta\prec \alpha}$ in
$L^2(\mathcal M,\tau)$.

\begin{lemma}
Suppose $x_1,\ldots,x_n$ is a free family of   random variables in
$\mathcal M$ with respect to the tracial state $\tau$. Let $\Sigma$,
$P_q(x_i)$ and $P_\alpha(x_1,\ldots, x_n)$ be as defined as above.
For each
$$\alpha= X_{i_1}^{j_1}X_{i_2}^{j_2}\cdots X_{i_m}^{j_m} \qquad \text { in } \ \Sigma_n$$ with $m\ge
1,  1\le i_1\ne i_2\ne \cdots \ne  i_m\le n$, we have
$$
P_\alpha(x_1,\ldots,x_n)= \prod_{k=1}^m P_{i_k,j_k}(x_{i_k}).
$$
\end{lemma}
\begin{proof} For each $\alpha= X_{i_1}^{j_1}X_{i_2}^{j_2}\cdots X_{i_m}^{j_m}$, let us denote
$\prod_{k=1}^m P_{i_k,j_k}(x_{i_k})$ by $Q_\alpha(x_1,\ldots,x_n)$.
By the definition, for any $1\le i_k\le n$ and $j_k\ge 1$, we know
that
$$P_{i_k,j_k}(x_{i_k}) = x_{i_k}^{j_k} - E_{i_k,j_k}(x_{i_k}^{j_k}), $$
where $E_{i_k,j_k}$ is the projection from $H=L^2(\mathcal M,\tau)$
onto the linear space spanned by $\{I, x_{i_k},
\ldots,x_{i_k}^{j_k-1}\}$ in $H$.  It is not hard to see that
$$
Q_\alpha(x_1,\ldots,x_n)=\prod_{k=1}^m P_{i_k,j_k}(x_{i_k}) =
x_{i_1}^{j_1}x_{i_2}^{j_2}\cdots x_{i_m}^{j_m}+ Q(x_1,\ldots, x_n),
$$ where $Q(x_1,\ldots, x_n)$  is a linear combination of
$\{x_\beta\}_{\beta\prec\alpha}$, i.e. $E_\alpha(Q(x_1,\ldots,x_n))=
Q(x_1,\ldots,x_n)$. Thus the subspace spanned by
$\{x_{\beta}\}_{\beta\prec \alpha}$ is equal to the subspace spanned
by $\{Q_{\beta}\}_{\beta\prec\alpha}$ in $H$.

On the other  hand, it follows  from the definition of the freeness
that $$\tau(Q^*_\beta(x_1,\ldots, x_n) Q_\alpha(x_1,\ldots,x_n))=0$$
for any $\beta\ne \alpha $ in $\Sigma_n$. It induces that
$Q_\alpha(x_1,\ldots,x_n)$ is orthogonal to the linear space spanned
by $\{Q_{\beta})(x_1,\ldots,x_n)\}_{\beta\prec\alpha}$ whence
$Q_\alpha(x_1,\ldots,x_n)$ is orthogonal to the linear space spanned
by $\{x_{\beta}\}_{\beta\prec \alpha}$. So
$E_{\alpha}(Q_\alpha(x_1,\ldots,x_n))=0$.

Hence
$$\begin{aligned}
0=E_{\alpha}(Q_\alpha(x_1,\ldots,x_n))&=E_{\alpha}(\prod_{k=1}^m
P_{i_k,j_k}(x_{i_k}) ) = E_{\alpha}(x_{i_1}^{j_1}x_{i_2}^{j_2}\cdots
x_{i_m}^{j_m})+E_\alpha( Q(x_1,\ldots, x_n))\\ &=
 P_\alpha(x_1,\ldots,x_n)-x_{i_1}^{j_1}x_{i_2}^{j_2}\cdots
x_{i_m}^{j_m}  +  Q(x_1,\ldots, x_n) \\
&= P_\alpha(x_1,\ldots,x_n)-\prod_{k=1}^m P_{i_k,j_k}(x_{i_k}).
\end{aligned}
$$
It follows that $P_\alpha(x_1,\ldots,x_n)= \prod_{k=1}^m
P_{i_k,j_k}(x_{i_k}).$
\end{proof}

\begin{lemma}
Denote, for every integer $q\ge 1$,
$$
s_{q}= \prod_{\alpha\in \Sigma, |\alpha|=q}
||P_{\alpha}(x_1,\ldots,x_n)\|_2^2.
$$
Then, we have
$$
\frac {s_{q+1}}{s_q^n}=\left ( \prod_{k=1}^n
\|P_{q+1}(x_k)\|_2^2\right ) \cdot \left (\prod_{k=1}^n
\prod_{j=1}^{q-1} (\|P_{ j}(x_k)\|_2^{2})^{ n^{^{j-1}}}\right ).
$$
\end{lemma}

\begin{proof}
Note the index set
$$
\Sigma_n=\{e\}\cup\{  X_{i_1}^{j_1}X_{i_2}^{j_2}\cdots X_{i_m}^{j_m}
: m\ge 1; 1\le i_1\ne i_2\ne \cdots \ne i_m\le n; j_1,j_2,\ldots,
j_m\ge 1 \}.
$$
We let, for each integer $q\ge 1$ and $1\le k \le n$,
$$
\begin{aligned}
A_q&=\{  X_{i_1}^{j_1}X_{i_2}^{j_2}\cdots X_{i_m}^{j_m}\in \Sigma_n
: m\ge 1; 1\le i_1\ne i_2\ne \cdots \ne i_m\le n; j_1+j_2+\cdots+
j_m=q \}\\
B_{q,k} &=\{ X_{i_1}^{j_1}X_{i_2}^{j_2}\cdots X_{i_m}^{j_m} \in A_q
: i_1=k\}\\
C_{q,k} &=\{ X_{i_1}^{j_1}X_{i_2}^{j_2}\cdots X_{i_m}^{j_m} \in A_q
: i_1\ne k\}.
\end{aligned}
$$
It is not hard to verify that
$$
A_q  =  B_{q,k}\cup C_{q,k} \qquad \text {for every } 1\le k\le n;$$
and $$\begin{aligned}
 A_q  = \bigcup_{k=1}^n B_{q,k} =\bigcup_{k=1}^n\left (\bigcup_{j=1}^{q-1} \left (
  X_k^j\cdot C_{q-j,k}  \right )\bigcup \{X_k^q\}
\right ),\end{aligned}
$$
where $$X_k^j\cdot C_{q-j,l}= \{ X_k^j\ \beta :  \beta\in C_{q-j,l}
\}.$$ Note that
$$
\{  X_k^j\cdot C_{q-j,k}, \{X_k^q\}\}_{1\le k\le n, \ 1\le j\le q-1
}
$$ is a collection of disjoint subsets of $A_q$. So
$$
  \begin{aligned}
s_{q}&= \prod_{\alpha\in \Sigma, |\alpha|=q}
||P_{\alpha}(x_1,\ldots,x_n)\|_2^2 =  \prod_{\alpha\in A_q}
||P_{\alpha}(x_1,\ldots,x_n)\|_2^2\\
& =  \prod_{\alpha\in \bigcup_{k=1}^n\left (\bigcup_{j=1}^{q-1}
\left (  X_k^j\cdot C_{q-j,k}  \right )\bigcup
\{X_k^q\} \right )} ||P_{\alpha}(x_1,\ldots,x_n)\|_2^2\\
& = \left (\prod_{k=1}^n\prod_{j=1}^{q-1}  \prod_{\alpha\in
X_k^j\cdot C_{q-j,k}  } ||P_{\alpha}(x_1,\ldots,x_n)\|_2^2 \right)
\left (\prod_{k=1}^n \| P_{k,q}(x_k)\|_2^2\right ).
  \end{aligned}
$$
Form the fact that the cardinality of the set $C_{q-j,k}$ is equal
to $(n-1)n^{q-j-1}$ and Lemma 4, it follows that
$$\begin{aligned}
  s_q&= \left (\prod_{k=1}^n\prod_{j=1}^{q-1}  \prod_{\beta\in
 C_{q-j,k}  }\left (\|P_{k,j}(x_k)\|_2^2\cdot ||P_{\beta}(x_1,\ldots,x_n)\|_2^2 \right )\right)
\left ( \prod_{k=1}^n\| P_{k,q}(x_k)\|_2^2\right )\\
&=\left (\prod_{k=1}^n\prod_{j=1}^{q-1}\left (
\|P_{k,j}(x_k)\|_2^{2(n-1)n^{^{q-j-1}}}\cdot \prod_{\beta\in
 C_{q-j,k}  } ||P_{\beta}(x_1,\ldots,x_n)\|_2^2  \right )\right)
\left (\prod_{k=1}^n \| P_{k,q}(x_k)\|_2^2\right )\\
&=\left(\prod_{k=1}^n\prod_{j=1}^{q-1}\|P_{k,j}(x_k)\|_2^{2(n-1)n^{^{q-j-1}}}\right)\left
(\prod_{k=1}^n\prod_{j=1}^{q-1}  \prod_{\beta\in
 C_{q-j,k}  } ||P_{\beta}(x_1,\ldots,x_n)\|_2^2  \right)
\left (\prod_{k=1}^n \| P_{k,q}(x_k)\|_2^2\right )\\
&=\left(\prod_{k=1}^n\prod_{j=1}^{q-1}\|P_{k,j}(x_k)\|_2^{2(n-1)n^{^{q-j-1}}}\right)\left
(\prod_{k=1}^n\prod_{j=1}^{q-1} \frac{ \prod_{\beta\in
 A_{q-j} } ||P_{\beta}(x_1,\ldots,x_n)\|_2^2}{\prod_{\beta\in
 B_{q-j,k} } ||P_{\beta}(x_1,\ldots,x_n)\|_2^2}  \right)
\left (\prod_{k=1}^n \| P_{k,q}(x_k)\|_2^2\right )\\
&=\left(\prod_{k=1}^n\prod_{j=1}^{q-1}\|P_{k,j}(x_k)\|_2^{2(n-1)n^{^{q-j-1}}}\right)\left
( \prod_{j=1}^{q-1} s_{j}^{n-1}  \right) \left (\prod_{k=1}^n \|
P_{k,q}(x_k)\|_2^2\right ).
  \end{aligned}
$$
Or,
$$
\frac {s_q}{\prod_{j=1}^{q-1} s_{j}^{n-1}}=
\left(\prod_{k=1}^n\prod_{j=1}^{q-1}\|P_{k,j}(x_k)\|_2^{2(n-1)n^{^{q-j-1}}}\right)
\left (\prod_{k=1}^n \| P_{k,q}(x_k)\|_2^2\right ).
$$
\end{proof}

The following lemma can be directly verified by combinatory method.
\begin{lemma}
Suppose that $\{c_q\}_{q=1}^\infty, \{d_q\}_{q=2}^\infty $ are two
sequences of positive numbers and $r>0$. If
$$
c_q-r\cdot \sum_{j=1}^{q-1} c_j = d_q, \qquad \text{   for  $q\ge
2$}
$$
then
$$
c_q= r(1+r)^{q-2}b_1 + r\sum_{j=2}^{q-1} (1+r)^{q-1-j}d_j+
d_q,\qquad \text { for } \ q\ge 2.
$$
\end{lemma}

Combining Lemma 5 and Lemma 6, we have the following.
\begin{lemma}
 $$ \begin{aligned}
     \ln s_1 &= \sum_{k=1}^n  \ln \|P_1(x_k)\|_2^2 \\
     \ln s_q & = (n-1) n^{q-2} \ln s_1 +  d_q +(n-1) \sum_{j=2}^{q-1}
     n^{q-1-j} d_j ,\qquad \text { for } \ q\ge 2,
  \end{aligned}$$
  where
  $$
 d_q =\left( (n-1)\sum_{k=1}^n\sum_{j=1}^{q-1} n^{q-j-1} \ln
 \|P_{k,j}(x_k)\|_2^2
 \right)+
\left (\sum_{k=1}^n \ln \| P_{k,q}(x_k)\|_2^2\right ).
  $$
\end{lemma}

\begin{proposition}
We have
$$
\begin{aligned}
   \ln s_1&=\sum_{k=1}^n \ln\|  P_1(x_k)\|_2^2\\
   \ln s_q&=\sum_{k=1}^n \ln \|P_{k,q}(x_k)\|_2^2+ 2(n-1)\sum_{j=1}^{q-1}\left ( n^{q-1-j}\sum_{k=1}^n \ln
   \|P_{k,j}(x_k)\|_2^2\right )\\ &\quad \qquad \quad  +(n-1)^2\sum_{j=1}^{q-2} \left ((q-2-j)
   n^{q-2-j}    \sum_{k=1}^n \ln \|P_{k,j}(x_k)\|_2^2 \right ),
\end{aligned}
$$ for all $q\ge 2$.
\end{proposition}
\begin{proof}
Let $$ C_{j} =  \sum_{k=1}^n \ln \|P_{k,j}(x_k)\|_2^2 , \qquad \text
{ for } j \ge 1.$$ Then by Lemma 7 we have
$$\begin{aligned}
    \ln s_q & = (n-1) n^{q-2} \ln s_1 +  d_q +(n-1) \sum_{j=2}^{q-1}
     n^{q-1-j} d_j ,\qquad \text { for } \ q\ge 2,\\
 d_q &=\left( (n-1) \sum_{j=1}^{q-1} n^{q-j-1} C_j
 \right)+
C_q.
\end{aligned}
$$
Thus,
$$
\begin{aligned}
   \ln s_q & = (n-1) n^{q-2} \ln s_1 + \left( (n-1) \sum_{j=1}^{q-1} n^{q-j-1} C_j
 \right)+
C_q\\
& +(n-1) \sum_{j=2}^{q-1}
     n^{q-1-j} \left (\left( (n-1) \sum_{m=1}^{j-1} n^{j-m-1} C_m
 \right)+
C_j \right ) ,\qquad \text { for } \ q\ge 2.\\
\end{aligned}
$$
So,
$$
  \begin{aligned}
      \ln s_q & = (n-1) n^{q-2} \ln C_1+ \left( 2(n-1) \sum_{j=2}^{q-1} n^{q-j-1}
      C_j+(n-1) n^{q-2}C_1
 \right)+
C_q\\
& \qquad \quad +(n-1) \sum_{j=2}^{q-1}
     n^{q-1-j} \left( (n-1) \sum_{m=1}^{j-1} n^{j-m-1} C_m
 \right)\\
 &=C_q + 2(n-1) \sum_{j=1}^{q-1} n^{q-j-1}
      C_j  +
      (n-1)^2\sum_{m=1}^{q-2}\sum_{j=m+1}^{q-1}n^{q-2-m}C_m\\
      & =C_q + 2(n-1) \sum_{j= 1}^{q-1} n^{q-j-1}
      C_j   +
      (n-1)^2\sum_{m=1}^{q-2}(q-m-2)n^{q-2-m}C_m,
  \end{aligned}
$$ where
$$
C_j= \sum_{k=1}^n \ln \|P_{k,j}(x_k)\|_2^2, \qquad \text { for }
j\ge 1.
$$
\end{proof}
\subsection{$n-$th entropy number}
\begin{definition}
Suppose $x$ is an   element in a free probability space $\mathcal M$
with a   tracial state $\tau$. For each $j\ge 1$, let $P_j(x)$ be
defined as in section 3.1. Then we define $n$-th  entropy number of
$x$
$$
\mathcal{E}_n(x) =\sum_{j=1}^\infty \frac {\ln \|P_j(x)\|_2^2}{n^j}.
$$
\end{definition}
By Lemma 1, we have
\begin{corollary}Suppose $x=x^*$ is a self-adjoint element in $\mathcal M$. For $n\ge 2$,
$$\mathcal E_n(x) = \frac{2(n-1)}n\sum_{j=1}^\infty \frac {\ln a_j}{n^j},$$
\end{corollary} where $a_1,a_2,\ldots$ are as defined in Lemma 1.
\begin{corollary}
Suppose $u$ is a unitary element in $\mathcal M$. For $n\ge 2$,
$$\mathcal E_n(x) = \frac{n-1}n\sum_{j=1}^\infty \frac {\ln (1-|\alpha_j|^2)}{n^j},$$
\end{corollary} where $\alpha_1,\alpha_2,\ldots$ are as defined in Lemma 2.

\subsection{Main result}
The following is the main result in the paper.
\begin{theorem}
Suppose $\langle\mathcal M,\tau\rangle$ is a free probability space.
Suppose $x_1,\ldots, x_n$$(n\ge 2)$ are  random variables in
$\mathcal M$ such that $x_1,\ldots, x_n$ are free with respect to
$\tau$. For each $q\ge 1$, let $D_q(x_1,\ldots,x_n)$ be defined as
in section 3.5.  Then we have
$$
\lim_{q\rightarrow\infty} \frac {\ln D_{q+1}(x_1,\ldots,x_n)}{q\cdot
n^{q}} = \frac {(n-1)}{n } \cdot \sum_{k=1}^n\mathcal E_n(x_k),
$$
where $\mathcal E_n(x_k)$ is $n-$th   entropy number of $x_k$ in
section 5.2.
\end{theorem}

\begin{proof} Let $\Sigma_n=\Bbb F_n^+ $ be the unital free semigroup
generated by $n$ generators $X_1,\ldots, X_n$ with lexicographic
order $\prec$. For each $\alpha $ in $\Sigma$, let
$P_\alpha(x_1,\ldots,x_n)$ be as defined in section 3.2.  For each
$q\ge 1$ and $1\le k\le n$, let $P_{k,q}(x_k)$ be as defined in
section 5.1. Let, for each $q\ge 1$, $$ s_{q}= \prod_{\alpha\in
\Sigma, |\alpha|=q} ||P_{\alpha}(x_1,\ldots,x_n)\|_2^2.
$$   By Proposition 3, we have
$$\begin{aligned}
\ln s_q&=\sum_{k=1}^n \ln \|P_{k,q}(x_k)\|_2^2+
2(n-1)\sum_{j=1}^{q-1}\left ( n^{q-1-j}\sum_{k=1}^n \ln
   \|P_{k,j}(x_k)\|_2^2\right )\\ &\quad \qquad \quad  +(n-1)^2\sum_{j=1}^{q-2} \left ((q-2-j)
   n^{q-2-j}    \sum_{k=1}^n \ln \|P_{k,j}(x_k)\|_2^2 \right
   ).\end{aligned}$$
Dividing by $qn^q$ on both side equation, we get
   $$\begin{aligned}
\frac 1 {qn^q}\ln s_q&=\frac 1 {qn^q}\sum_{k=1}^n \ln
\|P_{k,q}(x_k)\|_2^2+ \frac{2(n-1)}{qn}\sum_{j=1}^{q-1}\left ( n^{
-j}\sum_{k=1}^n \ln
   \|P_{k,j}(x_k)\|_2^2\right )\\ &\quad \qquad \quad  +\frac{(n-1)^2}{n^2}\sum_{j=1}^{q-2} \left
   (
   n^{ -j}    \sum_{k=1}^n \ln \|P_{k,j}(x_k)\|_2^2 \right )\\ &\qquad \qquad\quad+\frac{(n-1)^2}{qn^2}\sum_{j=1}^{q-2} \left (( -2-j)
   n^{ -j}    \sum_{k=1}^n \ln \|P_{k,j}(x_k)\|_2^2 \right ).
   \end{aligned}
$$
Since $\|P_{k,q}(x_k)\|_2\le \|x_k^q\|_2\le \|x_k\|^q$, we get $$
\frac 1 {qn^q}\sum_{k=1}^n \ln \|P_{k,q}(x_k)\|_2^2 , \qquad
\frac{2(n-1)}{qn^2}\sum_{j=1}^{q-1}\left ( n^{ -j}\sum_{k=1}^n \ln
   \|P_{k,j}(x_k)\|_2^2\right ) $$ and $$\frac{(n-1)^2}{qn^2}\sum_{j=1}^{q-2} \left (( -2-j)
   n^{ -j}    \sum_{k=1}^n \ln \|P_{k,j}(x_k)\|_2^2 \right )
$$ go   to $0$ as $q$ goes to $\infty$. Hence,
$$
\lim_{q\rightarrow\infty} \frac {\ln s_q}{q\cdot n^{q}} = \frac
{(n-1)^2}{n^2} \cdot \sum_{k=1}^n\sum_{j=1}^\infty\frac {
\ln\|P_{k,j}( {x_k}  )\|_2^2}{n^j}=\frac {(n-1)^2}{n^2}
\sum_{k=1}^n\mathcal E_n(x_k).
$$
Note that
$$
D_{q+1}(x_1,\ldots,x_n)=
\prod_{|\alpha|<q+1}\|P_\alpha(x_1,\ldots,x_n)\|_2^2= \prod_{j=1}^q
s_j.
$$ It follows that
$$\begin{aligned}
\lim_{q\rightarrow\infty}\frac {\ln
D_{q+1}(x_1,\ldots,x_n)}{qn^q}&=\lim_{q\rightarrow\infty}\frac { \ln
s_1+\ln s_2+\cdots+\ln s_q}{qn^q} \\
 &=\lim_{q\rightarrow\infty}\frac {\ln s_q}{qn^q-(q-1)n^{q-1}}
  =\lim_{q\rightarrow\infty}\frac {\ln s_q}{qn^q(1-(q-1)/(qn))}
  \\
&=\lim_{q\rightarrow\infty}\frac {\ln s_q }{qn^q}\cdot \frac n
{n-1}=\frac {(n-1) }{n } \sum_{k=1}^n\mathcal E_n(x_k).
\end{aligned}
$$
\end{proof}

By Corollary 1 and Corollary 2, we have the following results.
\begin{corollary}We assume the same notations as in Theorem 1. Suppose
$x_1,\ldots, x_n$ is a free family of self-adjoint elements in
$\mathcal M$. Then
$$
\lim_{q\rightarrow\infty} \frac {\ln D_{q+1}(x_1,\ldots,x_n)}{q\cdot
n^{q}} =
  \frac{2(n-1)^2}{n^2}\sum_{k=1}^n\sum_{j=1}^\infty \frac {\ln
 a_{k,j}}{n^j},
$$
where $a_{k,1}, a_{k,2},\ldots$ are the coefficients of Jacobi
matrix associated with $x_k$ (see Lemma 1).
\end{corollary}

\begin{corollary}We assume the same notations as in Theorem 1. Suppose
$u_1,\ldots, u_n$ is a free family of unitary elements in $\mathcal
M$. Then
$$
\lim_{q\rightarrow\infty} \frac {\ln D_{q+1}(u_1,\ldots,u_n)}{q\cdot
n^{q}} =
  \frac{  n-1 }{n }\sum_{k=1}^n\sum_{j=1}^\infty \frac {\ln
(1-|\alpha_{k,j}|^2)}{n^j},
$$
where $\alpha_{k,1}, \alpha_{k,2},\ldots$ are the Verblunsky
coefficients associated with $u_k$ (see Lemma 2).
\end{corollary}

\vspace{1cm} \small{}

\end{document}